\documentclass[12pt]{amsart}
\usepackage{amsfonts, amsthm, amssymb, amsmath, stmaryrd}
\usepackage{mathrsfs,array}
\usepackage{eucal,color,times,enumerate,accents}
\usepackage[all]{xy}
\usepackage{url}
\usepackage{enumitem}
\usepackage[pdftex,bookmarksnumbered,bookmarksopen]{hyperref}

\hypersetup{
  colorlinks,
  citecolor=red,
  linkcolor=blue,
  urlcolor=black}

    \newcommand\blfootnote[1]{%
  \begingroup
  \renewcommand\thefootnote{}\footnote{#1}%
  \addtocounter{footnote}{-1}%
  \endgroup
}
  
  \newcommand{\Addresses}{{
  \bigskip
  \footnotesize

\textsc{London School of Geometry and Number Theory, UCL, Department of Mathematics, Gower street, WC1E 6BT, London, UK}\par\nopagebreak
  \textit{E-mail address}, G.~Baldi: \texttt{gregorio.baldi.16@ucl.ac.uk}

}}

\usepackage{bbm}

\usepackage{leftidx}

\usepackage[bottom=3cm, top=2cm, left=2cm, right=2cm]{geometry}

\input xy
\xyoption{all}

\addtocontents{toc}{\setcounter{tocdepth}{1}}

\newcommand{\PP}{\mathbb{P}}
\usepackage{tikz-cd}
\usepackage{hyperref}
\hypersetup{colorlinks, linkcolor=blue}

\newtheorem{thm}{Theorem}[section]
 \newtheorem{cor}[thm]{Corollary}
 \newtheorem{lemma}[thm]{Lemma}
 \newtheorem{prop}[thm]{Proposition}
 \theoremstyle{definition}
 \newtheorem{defi}[thm]{Definition}
 \theoremstyle{remark}
 \newtheorem{rmk}[thm]{Remark}
  \newtheorem{question}[thm]{Question}
 
 \numberwithin{equation}{section}

\newcommand{\NS}{\operatorname{NS}}

\DeclareMathOperator{\End}{End}

\DeclareMathOperator{\Gal}{Gal}

\DeclareMathOperator{\Gl}{GL}

\newcommand{\Aut}{\operatorname{Aut}}

\newcommand{\N}{\mathbb{N}}

\newcommand{\Z}{\mathbb{Z}}
\newcommand{\Q}{\mathbb{Q}}

\newcommand{\Oo}{\mathcal{O}}

\newcommand{\C}{\mathbb{C}}

\newcommand{\Qbar}{\overline{\mathbb{Q}}}

\newcommand{\frv}{\text{Frob}_v}

\makeatletter
\def\@settitle{\begin{center}%
  \baselineskip14\p@\relax
  \bfseries
  \uppercasenonmath\@title
  \@title
  \ifx\@subtitle\@empty\else
     \\[1ex]\uppercasenonmath\@subtitle
     \footnotesize\mdseries\@subtitle
  \fi
  \end{center}%
}
\def\subtitle#1{\gdef\@subtitle{#1}}
\def\@subtitle{}
\makeatother

\begin{document}

\newcommand{\adjunction}[4]{\xymatrix@1{#1{\ } \ar@<-0.3ex>[r]_{ {\scriptstyle #2}} & {\ } #3 \ar@<-0.3ex>[l]_{ {\scriptstyle #4}}}}

\title{Local to global principle for the moduli space of K3 surfaces}\blfootnote{\emph{Date}. November 30, 2018.}
\author{Gregorio Baldi}

\begin{abstract}
Recently S. Patrikis, J.F. Voloch and Y. Zarhin have proven, assuming several well known conjectures, that the finite descent obstruction holds on the moduli space of principally polarised abelian varieties. We show an analogous result for K3 surfaces, under some technical restrictions on the Picard rank. This is possible since abelian varieties and K3s are quite well described by `Hodge-theoretical' results. In particular the theorem we present can be interpreted as follows: a family of $\ell$-adic representations that \emph{looks like} the one induced by the transcendental part of the $\ell$-adic cohomology of a K3 surface (defined over a number field) determines a Hodge structure which in turn determines a K3 surface (which may be defined over a number field).
\end{abstract}
\maketitle

\tableofcontents
\section{Introduction}
Let $X$ be an algebraic K3 surface defined over a number field $K$, and $\ell$ a rational prime. We consider $T_\ell(X_{\overline{K}})$ the transcendental part of the second $\ell$-adic cohomological group of $X$, i.e. $T_\ell(X_{\overline{K}})$ is the orthogonal complement of the image of the N\'{e}ron-Severi group of $X_{\overline{K}}=X \times _K \overline{K}$ in $H^2_{\text{et}}(X_{\overline{K}},\Q_\ell)$. It is a free $\Q_\ell$-module of rank $22-\rho$, where $\rho \in \{1,2,\dots, 20\}$ denotes the rank of the N\'{e}ron-Severi group of $X_{\overline{K}}$, usually called the (geometric) Picard rank of $X$. For every rational prime $\ell$, there is a continuous $\ell$-adic Galois representation of the absolute Galois group of $K$
\begin{displaymath}
\rho_{X,\ell}: \Gal(\overline{K}/ K) \to \Gl (T_\ell(X_{\overline{K}})).
\end{displaymath}
The family $\{\rho_{X,\ell}\}_\ell$ encodes many algebro-geometric properties of $X$ that can expressed in the language of representation theory.

The problem discussed in this note is motivated by the following question, which can be thought as a refinement of the Fontaine--Mazur conjecture (\cite[Conjecture 1]{MR1363495}), since it aims to describe the essential image of the $\ell$-adic realisations of K3 surfaces. 
\begin{question}\label{question1}
Given a family of $\ell$-adic representations of the absolute Galois group of a number field $K$, can we \emph{understand} if it is of the form $\{\rho_{X,\ell}\}_\ell$ for some K3 surface $X/K$ (possibly after a finite field extension $L/K$)?
\end{question}
Galois representations coming from the cohomology of (smooth projective) varieties satisfy a number of constraints that are best understood when formulated in the language of $p$-adic Hodge theory. Indeed a deep result of Faltings, \cite[Chapter III, Theorem 4.1]{MR924705}, shows that if $Y$ is a proper and smooth variety over a $p$-adic field $K_v$, then the cohomology groups $H_{\text{et}}^i(Y\times \overline{K_v},\Q_p)$ give rise to de Rham representations. For an accessible introduction to the notions of $p$-adic Hodge theory (such as de Rham representations and their Hodge--Tate weights), we refer the reader to \cite{comparison} (in particular sections I.2 and II.6) and the monograph \cite{fontouyang}.

The analogous of Question \ref{question1} for abelian varieties has been addressed and solved, assuming several well known conjectures, in \cite[Theorem 3.1]{MR3544295} (at least for abelian varieties with endomorphism ring equal to $\Z$). More precisely the authors proved the following: 
\begin{thm}[Patrikis, Voloch, Zarhin] \label{3authors}
Assume the Hodge, Tate, Fontaine--Mazur and the semisimplicity conjecture. Let 
\begin{displaymath}
\{\rho_\ell:\Gal(\overline{K}/ K)\to \Gl_{2N}(\Q_\ell)\}_\ell
\end{displaymath}
be a weakly compatible family (in the sense of Definition \ref{weakcomp}) of $\ell$-adic representations such that:
\begin{itemize}
\item[i)] For some prime $\ell_0$, $\rho_{\ell_0}$ is de Rham at all places of $K$ above $\ell_0$;
\item[ii)] For some prime $\ell_1$, $\rho_{\ell_1}$ is absolutely irreducible;
\item[iii)] For some prime $\ell_2$ and at least one place $v$ above $\ell_2$, ${\rho_{\ell_2}}_{ | \Gal(\overline{K}_v/ K_v)}$ is de Rham with Hodge--Tate weights $-1,0$ each with multiplicity $N$, where $K_v$ denotes the completion of $K$ at the place $v$.
\end{itemize}
Then there exists a $N$-dimensional abelian variety $A$ defined over $K$ such that $\rho_\ell\cong V_\ell(A)$ for all $\ell$, where $V_\ell(A)$ denotes the rational $\ell$-adic Tate module of $A$ with its natural Galois action.
\end{thm}
Notice that conditions i) and iii) are satisfied by the cohomology of every abelian variety, and condition ii) holds for the generic abelian variety.

It is reasonable to expect a result of the same fashion for varieties whose geometry is well captured from cohomological invariants. For example the above theorem can not tell the difference between a curve of genus $N>1$ and its Jacobian (see \cite[Section 5]{MR3544295} for a more detailed discussion about this). From this point of view K3 surfaces (and hyperk\"{a}hler varieties) are very similar to abelian varieties. Indeed, over the complex numbers, they enjoy a Torelli type theorem (\cite{torelli}) and the surjectivity of the period map (\cite{todorov}), see Proposition \ref{lattices} for precise statements.

\subsection{Main results}
In Section \ref{proof} we will prove the following, which is the main theorem of the paper.
\begin{thm}\label{mainthm}
Assume the Tate, Fontaine--Mazur, and the Hodge conjectures. Let $\rho$ be a natural number such that such that $2< 22-\rho \leq 19$ and let 
\begin{displaymath}
\{ \rho_\ell:\Gal(\overline{K}/ K) \to \Gl_{22-\rho}(\Q_\ell)\}_\ell,
\end{displaymath}
be a weakly compatible family of $\ell$-adic representations satisfying the following conditions: 
\begin{enumerate}
\item For some prime $\ell_0$, $\rho_{\ell_0}$ is de Rham at all places of $K$ above $\ell_0$;
\item For some prime $\ell_1$, $\rho_{\ell_1}$ is absolutely irreducible;
\item For some prime $\ell_2$ and at least one place $v$ above $\ell_2$, ${\rho_{\ell_2}}_{ | \Gal(\overline{K}_v/ K_v)}$ is de Rham with Hodge--Tate weights $0,1,2$, with multiplicities, respectively, $1,20-\rho,1$.
\end{enumerate}
Then there exists a K3 surface $X$ defined over a finite extension $L/K$ with geometric Picard rank $\rho$, such that the restriction of $\rho_\ell$ to $\Gal (\overline{L} / L)$ is isomorphic to $T_\ell(X_{\overline{L}})$ for all $\ell$.
\end{thm}
The proof will actually show something more: there exists a motive $M$ defined over $K$ inducing the representations $\rho_\ell$ and a finite extension $L/K$, such that the base change of $M$ to $L$ is isomorphic to the transcendental part (in the sense of Section \ref{decmotivsurface}) of the motive of a K3 surface defined over $L$. It is not clear whether or not the extension $L/K$ is needed, more about this is discussed in Section \ref{fieldext}.

In the proof, from the motive $M/K$, we will first produce a complex (algebraic) K3 surface and descend it to a number field. This is shown in the last section, which may be of independent interest. For a complex K3 surface $X$, and an element $\sigma \in  \Aut(\C / \Qbar)$ we set
\begin{displaymath}
\leftidx{^\sigma}{X} := X \times_{\C, \sigma}\C
\end{displaymath}
for the conjugate of $X$ with respect to $\sigma$. Let $T(X)_\Q$ be the rational polarised Hodge structure given by the transcendental part of the $H^2(X(\mathbb{C}),\Q)$, i.e. the orthogonal complement of the image of $\NS(X)\otimes \Q$ in $H^2(X,\Q)$. We have
\begin{thm}\label{main2}
Let $X/\C$ be a K3 surface such that 
\begin{displaymath}
T(X)_\Q \cong T(\leftidx{^\sigma}{X})_\Q  \ \ \text{   for all   } \sigma\in  \Aut(\C / \Qbar),
\end{displaymath}
where the isomorphism is an isomorphism of rational polarised Hodge structures. Then $X$ admits a model defined over a number field, i.e. there exists a number field $L\subset \C$ and a K3 surface $Y/L$, such that $Y \times_{L}\C$ is isomorphic to $X$.
\end{thm}
The above condition can be thought as an \emph{isogeny} relation between $X$ and its $\Aut(\C / \Qbar)$-conjugates. With this interpretation the theorem is analogous of the descent for abelian varieties established in \cite[Lemma 3.6]{MR3544295}: Let $A/\C$ be an abelian variety such that all its $\Aut(\C / \Qbar)$-conjugates are isogenous to $A$, then $A$ descends to a number field. Theorem \ref{main2} will follow from a general criterion proven by Gonz\'{a}lez-Diez (see Lemma \ref{descenttoqbar}) and the fact that there are only finitely many complex K3 surfaces with given transcendental lattice (see Lemma \ref{countable}). It is quite different from the proof of \cite[Lemma 3.6]{MR3544295} and can be used also to reprove such result (as explained in Remark \ref{isogenyandabvar}).

In the next remark we show that Theorem \ref{main2} implies that complex K3 surfaces with complex multiplication (or CM, for brevity) admit models over number fields. By complex multiplication we intend that the Mumford-Tate group associated to the Hodge structure $T(X)_\Q$ is commutative; for example every K3 surface of geometric Picard rank $20$ has such property (cf. \cite[Remark 3.10 (page 54)]{bookk3}). If $X/K$ is a CM K3 surface, it is a consequence of the Kuga-Satake construction and Deligne's work on absolute Hodge cycles (see for example \cite[Chapter II, Proposition 6.26 and Corollary 6.27]{MR654325}), that the $\ell$-adic monodromy associated to ${T(X_{\overline{K}})}_{\Q_\ell}$ is commutative as well. 
\begin{rmk}
One can check that complex K3 surfaces with complex multiplication satisfy the condition of the above theorem. Therefore they admit a model over a number field, reproving a result originally due to Pjatecki\u{\i}-\v{S}apiro and \v{S}afarevi\v{c}, see \cite[Theorem 4]{MR0335521}. For example, when $X/\C$ has maximal Picard rank, it is enough to notice that, thanks to the comparison between Betti and \'{e}tale cohomology, the quadratic forms associated to $T(X)$ and $T(\leftidx{^\sigma}{X})$ are in the same genus.
\end{rmk}

Finally we point out that recently C. Klevdal, in \cite{2018arXiv180702579K}, considered an analogue of Theorem \ref{mainthm} for K3 surfaces of Picard rank 1. Therefore only the case of geometric Picard rank two and twenty are left out of the picture. We hope to come back to this in the future. We refer the reader to \cite[Section 1.1]{2018arXiv180702579K} for a comparison between the two results and a precise statement of Klevdal's result.

\subsection{Comments on the assumptions of Theorem \ref{mainthm}}\label{comments}
As remarked above, the cohomology of K3 surfaces (defined over local fields) gives rise to de Rham representations. In particular condition (3) is satisfied if there exists a K3 surface $X_v/K_v$ of geometric Picard rank $\rho$ and ${\rho_{\ell_2}}_{| \Gal(\overline{K}_v/K_v)}$ is isomorphic to the representation induced by ${T_{\ell_2} (X_{\overline{K}_v} )}$\footnote{In condition iii) of Theorem \ref{3authors}, the Hodge--Tate weights are $-1,0$ since they want to relate ${\rho_{\ell_2}}_{| \Gal(\overline{K}_v/K_v)}$ to the Tate module of an abelian variety, which is the dual representation attached to the $H^1$. This explain the change of sings between the two theorems.}. So (1) and (3) are necessary conditions for the existence of the K3 surface the theorem aims to prove.

Condition (2), which appears also in Theorem \ref{3authors} as ii), is crucial to make the argument work (see Lemma \ref{lemma}) but it is not satisfied by every K3 surface (even if it holds for the generic K3 surface of Picard rank $\rho$). Klevdal considers the representations arising from the full $H^2$ of K3 surface of geometric Picard rank 1 and works with a different irreducibility condition (see condition (3) in \cite[Theorem1.1]{2018arXiv180702579K}). His conditions, which are satisfied by the cohomology of the generic K3 surface, imply that that the representation attached to the motive $M$ splits as a sum of the trivial representation and an absolutely irreducible one, and then works with the latter.

The absolutely irreducibility of condition (2) can not be weakened to require only the irreducibility of the $\ell$-adic Galois representations, as the beginning of section \ref{proof} shows. Notice also that the $\ell$-adic and Betti realisations of simple motives may be reducible and the irreducibility of the Hodge structure associated to a motive wont imply the irreducibility of the Galois representations attached to it. For example, if $X$ is a K3 surface defined over a number field $K$, $T(X)_\Q$, as Hodge structure, is always irreducible (see \cite[Chapter 3, Lemma 2.7 and Lemma 3.1]{bookk3}), but the Galois representations ${T_\ell(X_{\overline{K}})}$ may be reducible. This happens, for some $\ell$, whenever $X$ has complex multiplication.

The restriction on the Picard rank is due to the way we obtain a K3 surface from a polarised Hodge structure of K3 type (through Proposition \ref{lattices}). As we observed, the conditions in the theorem require the representations to have non-commutative image (unless they have values in $\Gl_1(\Q_\ell)$). Therefore our approach can not deal with case of Picard rank 20 (as Theorem \ref{3authors} excludes abelian varieties with complex multiplication).

\subsection{How to get rid of the finite field extension  \texorpdfstring{$L/K$}{L/K}}\label{fieldext} A positive answer to the following would allow to take $L=K$ in Theorem \ref{mainthm}.
\begin{question}\label{question}
Assume the conjectures as in Theorem \ref{mainthm}. Let $M$ be a simple motive defined over some number field $K$. Assume there exists a finite (Galois) extension $L/K$ and a K3 surface $Y/L$ such that $M_ L$ is isomorphic to the transcendental part of the motive of $Y_L$ (in the sense of Section \ref{decmotivsurface}). Is there a K3 surface $X$ defined over $K$ such that the transcendental part of the motive associated to $X$ is isomorphic to $M$.
\end{question}
In the case of abelian varieties (see the discussion before \cite[Proof of Lemma 3.6]{MR3544295}), the authors give an affirmative answer. The proof works by considering the Weil restriction to $K$ of the abelian variety $Y_L$ and, using Frobenius reciprocity, producing an endomorphism of it whose image corresponds to a $K$-abelian variety with the desired property. 

Unfortunately the argument does not apply to K3 surfaces and it is not clear if this field extension is necessary or not. Also in the situation considered by Klevdal a field extension is required, but, under an extra condition \cite[Condition (5) in Theorem 1.1]{2018arXiv180702579K} it is possible to give a bound on the degree of $L/K$. 
\subsection{Examples and applications}
We explain how to obtain examples of Galois representations to which Theorem \ref{mainthm} applies, without writing down a K3 surface, and how it produces a \emph{phantom isogeny class} of K3 surfaces, in analogy with the phantom isogeny class of abelian varieties defined by Mazur in \cite[page 38]{mazur}. 

Let $Y/K$ be a smooth projective variety defined over a number field $K$ such that, for some $i$, the Hodge decomposition induced on the primitive cohomology $H_{\text{prim}}^{2i}(Y(\C),\Q)$ \emph{looks like} the one of a K3 surface. More precisely its Hodge numbers are all zero but $h^{i,i}$ and $h^{p,q}=1$ for a unique, up to reordering, pair $(p,q)$ with $p+q=2i$ and the transcendental part of the $H^{i,i}$ has positive dimension less or equal than 18. Examples are provided by the $H^4$ of cubic fourfolds (where the Hodge numbers are $h^{0,4}=0, h^{1,3}=1, h^{2,2}=21$) and many varieties with $h^{2,0}=1$.

Consider the family of representations attached to the transcendental part of the $H_{\text{et}}^{2i}(Y_{\overline{K}},\Q_\ell (i-1))$, simply denoted by $T_\ell(Y_{\overline{K}})$, where we considered a Tate twist by $(i-1)$ to obtain the weight of an $H^2$:
\begin{displaymath}
\{\rho_{Y,\ell}: \Gal(\overline{K}/ K) \to \Gl (T_\ell(Y_{\overline{K}}))\}_{\ell}.
\end{displaymath}
The geometric origin of such representations implies that $\{\rho_{Y,\ell}\}_\ell$ is a weakly compatible system (cf. Remark \ref{deligne}) and that condition (1) of Theorem \ref{mainthm} is satisfied. The assumptions on the Hodge decomposition of the $H^{2i}(Y(\C),\Q)$ imply that condition (3) is satisfied. 

Assume now that $Y$ is such that also condition (2) is satisfied, i.e. for some prime $\ell_1$, $\rho_{\ell_1}$ is absolutely irreducible (such condition is satisfied by the generic cubic fourfold). Theorem \ref{mainthm}, after a finite extension $L/K$, associates a K3 surface $X/L$ to $Y$, with an isomorphism of $\Gal(\overline{L}/L)$-representations between $\rho_{Y,\ell}$ and $\rho_{X,\ell}$. Such K3 surface need not to be unique, and we think about K3s satisfying such condition as a \emph{phantom isogeny class}. The existence of such K3s could greatly simplify the study of Galois representations attached to such $Y$s. It would be interesting to construct them (over a number field!) without assuming any conjectures. For a survey explaining how K3 surfaces can help the study of the geometry of cubic fourfolds we refer the reader to \cite{MR3618665}. In particular in \cite[Section 3]{MR3618665} it is discussed how to associate K3 surfaces to \emph{special} cubic fourfolds via Hodge theoretical methods.

Theorem \ref{mainthm} has also the following amusing consequence, purely expressed in the $\ell$-adic language.
\begin{cor}
Assume the Tate, Fontaine--Mazur, and the Hodge conjectures. Let 
\begin{displaymath}
\{ \rho_\ell:\Gal(\overline{K}/ K) \to \Gl_{22-\rho}(\Q_\ell)\}_\ell,
\end{displaymath}
be a family of $\ell$-adic representations as considered in Theorem \ref{mainthm}. Then there exists a reductive group $G/\Q$ such that, after a finite extension $K'/K$, for every $\ell$, the image of $\Gal(\overline{K}/K')$ via $\rho_\ell$ has finite index in $G( \Q_\ell)$ and the index is bounded when $\ell$ varies.
\end{cor}
\begin{proof}
Theorem \ref{mainthm} shows that the weakly compatible system $\{ \rho_\ell\}_\ell$, up to replacing $K$ with a larger number field is associated to a K3 surface $X$. By a result of Serre there exists a finite extension $K'/K$ such that the $\ell$-adic monodromy of the ${\rho_\ell}_{| \Gal(\overline{K}/K')}$ is connected for every $\ell$ (see \cite[Section 1.1]{cadoretmoonen} for precise references). Let $G$ be the Mumford-Tate group of the K3 surface $X_\C$. The combination of the Tate and Hodge conjectures implies that, for every $\ell$, the image of ${\rho_\ell}_{| \Gal(\overline{K}/K')}$ is contained in $G(\Q_\ell)$ as a subgroup of finite index. Thanks to \cite[Theorem 6.6]{cadoretmoonen}, which is peculiar to K3 surfaces and abelian varieties, we have furthermore that the index is bounded independently from $\ell$.
\end{proof}
\begin{rmk}
In the proof of Theorem \ref{mainthm} it is first produced a motive $M$ whose $\ell$-adic realisations induce the family $\{\rho_\ell\}_\ell$. This weaker conclusion is not enough to obtain the corollary. Indeed the proof uses the \emph{Integral Mumford-Tate conjecture} which is known to follow from the classical Mumford-Tate conjecture, thanks to the work of Cadoret and Moonen, only for Galois representations attached to K3 surface and abelian varieties. More details about this can be found in \cite[Sections 1 and 2]{cadoretmoonen}. 
\end{rmk}

To conclude the introduction we point out that Theorem \ref{mainthm} can be interpreted in the setting of anabelian geometry and the section conjecture for the moduli space classifying primitively polarized K3 surfaces of degree $2d$. For more about this we refer the reader to the introduction of \cite{MR3544295} and the second part of Section 1.1 in \cite{2018arXiv180702579K}.

\subsection{Outline of paper}
In Section \ref{statement} we review the notion of weakly compatible systems, explain the formalism of motives we use and how to interpret the conjectures we assume. In Section \ref{proof} we prove the first main result (assuming Theorem \ref{main2}). The beginning of the proof closely follows the proof of \cite[Theorem 3.1]{MR3544295}, and we only recall the main steps. The last section, which is independent from the previous ones, proves Theorem \ref{main2}.

\subsection{Notations} By K3 surface $X/K$ we mean a complete smooth $K$-variety of dimension two such that $\Omega^2_{X/K}\cong \Oo_{X}$ and $H^1(X,\Oo_X)=0$. In this note $K$ will always denote a subfield of $\C$. For a complete overview of the theory of K3 surfaces we refer to the book \cite{bookk3}. We will make free use of the following standard notations:
\begin{itemize}
\item We denote by $(\Lambda_{\text{K3}},q)$ the K3-lattice, where $q$ is the quadratic form: it is the unique even unimodular lattice of signature $(3,19)$ (i.e. $E_8(-1)^2 \oplus U^3$);
\item We write \emph{PHS} as an acronym for \emph{rational} polarised Hodge structure, and $\Z$-PHS for \emph{integral} polarised Hodge structure (in particular we require that the underlying $\Z$-module is torsion free). Morphism in the category of PHS are maps of Hodge structure preserving the induced pairing;
\item By \emph{Hodge structure of K3 type} we mean an irreducible PHS of weight two such that $h^{2,0}=h^{0,2}=1$. In other references the irreducibility is not part of the definition, and they refer to \emph{irreducible PHS of K3 type};
\item Let $X$ be a complex K3 surface, we denote by $T(X)_\Q$ the transcendental part of the $H^2(X(\mathbb{C}),\Q)$, i.e. the orthogonal complement of $\NS(X)\otimes \Q \subset H^2(X(\mathbb{C}),\Q)$. It is a Hodge structure of K3 type (the irreducibility was first established in \cite{zarhinhodge});
\item Analogously, if $X$ is a K3 surface defined over a number field $K$, we define ${T_\ell(X_{\overline{K}})}$ to be the orthogonal complement of the image of $\NS(X_{\overline{K}})\otimes \Q_\ell$ in $H^2_{\text{et}}(X_{\overline{K}},\Q_\ell )$ with respect to the cup product in $\ell$-adic cohomology. 
\end{itemize}

\subsection{Acknowledgements}
It is a pleasure to thank Domenico Valloni for reading a draft of this paper and Alexei Skorobogatov for useful discussions regarding the theory of K3 surfaces. We are grateful to an anonymous referee whose precious comments improved the exposition and Proposition \ref{lattices}.

\subsection{Funding}
This work was supported by the Engineering and Physical Sciences Research Council [EP/ L015234/1], the EPSRC Centre for Doctoral Training in Geometry and Number Theory (The London School of Geometry and Number Theory), University College London.

\section{Weakly compatible systems and motives}\label{statement}
Let $K$ be a number field, $\overline{K}$ a fixed algebraic closure, $\Gal(\overline{K}/ K)$ its absolute Galois group. Given a rational prime $\ell$, we denote by $\Sigma_\ell$ the set of places of $K$ dividing $\ell$. From the beginning we also choose an embedding $\overline{K}\hookrightarrow \mathbb{C}$; this will be used whenever we compare the Betti cohomology to the $\ell$-adic one. If $v$ is a place of $K$ we write $K_v$ for the local field obtained completing $K$ at $v$ and $\Gal(\overline{K}_v/ K_v)$ for its absolute Galois group. 

\subsection{Systems of \texorpdfstring{$\ell$}{l}-adic Galois representations}
Let $\rho \in \N$ be such that $22-\rho > 0$. Consider a family of continuous $\ell$-adic Galois representation for every rational prime $\ell$
\begin{displaymath}
\{ \rho_\ell: \Gal(\overline{K}/ K) \to \Gl_{22-\rho}(\Q_\ell)\}_\ell.
\end{displaymath}
The definition of weakly compatible families presented is orginally due to Serre, who called them \emph{strictly compatible} on page I-11 of the book \cite{serrebook}.
 \begin{defi}[Weakly compatible]\label{weakcomp}
A family $\{\rho_\ell:\Gal(\overline{K}/ K)\to \Gl_{22-\rho}(\Q_\ell)\}_\ell$ is \emph{weakly compatible} if there exists a finite set of places $\Sigma$ of $K$ such that
\begin{itemize}
\item[(i)] For all $\ell$, $\rho_\ell$ is unramified outside the union of $\Sigma$ and the places of $K$ dividing $\ell$;
\item[(ii)] For all $v \notin \Sigma \cup \Sigma_\ell$, denoting with $\frv$ a Frobenius element at $v$, the characteristic polynomial of $\rho_\ell(\frv)$ has rational coefficients and it is independent of $\ell$.
\end{itemize}
\end{defi}
Recall that $\rho_\ell$ is said to be unramified at a place $v$ of $K$ if the image of the inertia at $v$ is trivial. If $\rho_\ell$ is attached to the $\ell$-adic cohomology of a smooth proper variety defined over a number field, the smooth and proper base change theorems, see for example \cite[I, Theorem 5.3.2 and Theorem 4.1.1]{MR463174}, imply that $\rho_\ell$ is unramified at every place $v\notin \Sigma_\ell$ such that $X$ has good reduction at $v$.
\begin{rmk}\label{deligne}
Deligne's work on the Weil conjectures (\cite[Theorem 1.6]{weil1}) the smooth and proper base change theorems imply that the $\ell$-adic representations attached to an $H_{\text{et}}^i(X_{\overline{K}},\Q_\ell(j))$ form a weakly compatible system, whenever $X$ is a smooth projective variety defined over a number field $K$.
\end{rmk}

\subsection{Tannakian and motivical formalism}\label{tannkaian} In this section we recall the motivical formalism and explain the conjectures appearing in the main theorem.

For any field $E$ of characteristic zero, we denote by $\mathcal{M}_{K,E}$ the category of pure homological motives over $K$ with coefficients in $E$. As recalled in \cite[Lemma 3.2]{MR3544295}, assuming the Tate conjecture, it is a semisimple category and it is equivalent to the category $\text{Rep} (\mathcal{G}_{K,E})$ for some pro-reductive group $\mathcal{G}_{K,E}$ (choosing an $E$-linear fibre functor).

We fix a family of embeddings $\iota_\ell : \overline{\Q}\to \overline{\Q}_\ell$ and write 
\begin{displaymath}
H_\ell : \mathcal{M}_{K,E}\to \text{Rep}_{\overline{\Q}_\ell}(\Gal(\overline{K}/ K))
\end{displaymath}
for the $\ell$-adic realisation functor associated to $\iota_\ell$.

What we need to know about the conjectures assumed in Theorem \ref{mainthm} can be found in \cite{MR3544295}, especially in Lemma 3.2, 3.3 and 3.4. A careful reader may notice that, in \cite[Theorem 3.1]{MR3544295} (stated as Theorem \ref{3authors} in the text), the semisimplicity conjecture it is also assumed. We can avoid this, since it has recently been proved that such conjecture follows from others. More precisely the following is \cite[Theorem 1]{moontate}:
\begin{thm}[Moonen]\label{moon}
The Tate conjecture implies the semisimplicity conjecture, i.e. if the functor $H_\ell$ is fully faithful then it takes value in the category of semisimple Galois representations.
\end{thm}
In what follows, we will therefore freely use the fact that the $\ell$-adic realisations of a motive are semisimple as Galois modules.

The main point is that the Fontaine--Mazur conjecture  (\cite[Conjecture 1]{MR1363495}), together with the Tate conjecture (\cite[Conjecture $T^j(X)$ for every $j$ and every $X$ (page 72)]{MR1265523}), describes the image of the $\ell$-adic realisation functors, as the following lemma explains (the Hodge conjecture will be used in the proof to compare the $\ell$-adic realisation functor with the Betti one).
\begin{lemma}[Lemma 3.3 of \cite{MR3544295}]\label{lemma}
Assume the Tate and the Fontaine--Mazur conjecture. Let $r_\ell: \Gal(\overline{K}/K) \to \Gl_N(\Q_\ell)$ be an irreducible geometric Galois representation. Then there exists an object $M \in\mathcal{M}_{K,\Qbar}$ such that
\begin{displaymath}
r_\ell \otimes \overline{\Q}_\ell \cong H_\ell(M) \in \operatorname{Rep}_{\overline{\Q}_\ell}(\Gal(\overline{K}/K)).
\end{displaymath}
\end{lemma}

\subsection{The motive of a surface}\label{decmotivsurface}
Let $M=h(X)$ be the motive of a (smooth projective connected) surface. The class of the diagonal $\Lambda \in \text{Corr}^0(X \times X)$ can be written as 
\begin{displaymath}
[\Lambda] = \sum_i \pi^i \in H^4(X\times X).
\end{displaymath}
Since the decomposition is algebraic, in the sense that each $\pi^i$ can be seen as the class of some orthogonal projector $\pi^i $ in $\text{Corr}^0(X \times X)$, we can decompose the motive $M$ as follows (Chow-K\"{u}nneth decomposition):
\begin{displaymath}
M= \mathbbm{1} \oplus h^1(X)\oplus h^2(X)\oplus h^3(X)\oplus \mathbb{L}^2,
\end{displaymath}
where $\mathbbm{1}$ is the motive of a point, $\mathbb{L}$ is the Lefschetz motive (defined by the equation $h(\PP^1_k)=\mathbbm{1} \oplus \mathbb{L}$), and $h^i(X)=(X,\pi^i,0)$.

Moreover, in \cite[Prop 2.3]{kahnmurrepedrini}, it is explained that there exists a unique splitting 
\begin{displaymath}
\pi^2= \pi^2_{alg}+\pi^2_{tr}
\end{displaymath}
inducing a refined Chow-K\"{u}nneth decomposition for the motive $M$:
\begin{displaymath}
h^2(X)= \left( h^2_{alg}(X)\oplus t^2(X)\right)
\end{displaymath}
where $h^2_{alg}(X)=(X, \pi^2_{alg},0)$ and $ t^2(X)=(X,\pi^2_{tr},0)$.

In this note we will be interested in the case of a K3 surface, so, from now on we will consider just the weight-two part of the motive of $X$. In particular the Betti realisation satisfies the following relation:
\begin{displaymath}
H_B(h^2_{alg}(X)\oplus t^2(X))=\NS(X)_\Q \oplus T(X)_\Q.
\end{displaymath}

\section{Proof of Theorem \ref{mainthm}}\label{proof}
The proof of the Theorem begins like the argument in \cite{MR3544295}, so we only recall the main steps. Thanks to the Tate, Fontaine--Mazur (and semisimplicity) conjectures the essential image of the $\ell$-adic realisation functor from the category of motives over $K$ with coefficients in $\overline{\Q}$ can be described explicitly, as in Lemma \ref{lemma}. In particular, choosing a place $\ell_0$ as in (1), there exists a representation of the group $\mathcal{G}_{K,E}$ (as explained in section \ref{tannkaian})
\begin{displaymath}
\rho : \mathcal{G}_{K,E}\to \Gl_{22-\rho,E}
\end{displaymath}
for some number field $E$, such that $H_{\ell_0}(\rho)\cong \rho_{\ell_0}\otimes \overline{\Q}_{\ell_0}$. Since the family $\{\rho_\ell\}_\ell$ is weakly compatible, and we assumed $\rho_{\ell_1}$ to be absolutely irreducible, it follows that $\rho$ induces every $\rho_\ell$. This allows to read the assumptions imposed at some prime $l_i$ in every $\rho_\ell$. Finally, since every $\rho_\ell$ is a representation with $\Q_\ell$ coefficients (rather than with coefficients in the completions of $E$), hence \cite[Lemma 3.4]{MR3544295} guarantees that the representation $\rho$ can be defined over $\Q$.

To summarize, we know that the compatible family $\{\rho_\ell\}_\ell$ arises from a representation
\begin{displaymath}
\rho: \mathcal{G}_K\to \Gl _{22-\rho, \Q},
\end{displaymath}
or, in equivalent terms, from a motive $M\in \mathcal{M}_K$ of rank $22-\rho$. By construction $M$ is also absolutely simple and $\End(\rho)=\Q$.

By hypothesis there exists a prime $\ell_2$ and a place $v$ dividing $\ell_2$ such that ${\rho_{\ell_2}}_{ | \Gal(\overline{K}_v/ K_v)}$ is de Rham with Hodge--Tate numbers equal to those of the transcendental lattice of a K3 surface of rank $\rho$.  Denote by $H_{dR}: \mathcal{M}_K \to \text{Fil}_K$ the de Rham realisation functor into the category of filtered $K$-vector spaces. From the comparison theorem between de Rham and \'{e}tale cohomology\footnote{Such comparison was conjectured by Fontaine in \cite[Conjecture A.6]{MR657238} and proved by Faltings in \cite{MR1463696}.} we have
\begin{displaymath}
H_{dR}(M)\otimes_K \mathbb{B}_{dR, K_v}\cong H_{\ell_2}(M) \otimes_{\Q_{\ell_2}} \mathbb{B}_{dR, K_v}
\end{displaymath}
where $\mathbb{B}_{dR, K_v}$ is the de Rham period ring over $K_v$. The fact that the above isomorphism is compatible with the filtration and the Galois action, the definition of $D_{dR,K_v}$ and the fact that $\mathbb{B}_{dR}^{\Gal(\overline{K}_v/K_v)}=K_v$, imply that
\begin{equation}	\label{weights}
H_{dR}(M)\otimes_K K_v \cong D_{dR,K_v}(H_{\ell_2}(M)).
\end{equation}
We write $M_{| \C}$ for the base change of $M\in \mathcal{M}_K$ in the category $\mathcal{M}_{K,\C}$ (we fixed from the beginning an embedding of $K$ into the complex numbers). Recall the Betti-de Rham comparison isomorphism:
\begin{displaymath}
H_{dR}(M)\otimes _K \mathbb{C}\cong H_B(M_{| \C}) \otimes _\Z \mathbb{C}.
\end{displaymath}
By \ref{weights}, $H_B(M_{| \C})$ is a polarizable rational Hodge structure of weight two and with Hodge numbers $1-(20-\rho)-1$. 

\begin{rmk}
While the previous part of the proof works in general (and closely follows the beginning of the proof of \cite[Theorem 3.1]{MR3544295}), from now on we will use in a substantial way the condition on the Hodge--Tate weights to produce a K3 surface. Our aim is to use $H_B(M_{| \C})$ to produce a period and then a K3 surface, invoking the surjectivity of the period map of Todorov. To apply this strategy we need the conjunction of the Tate and the Hodge conjecture, so that we can deduce properties of the Hodge structure from the properties imposed on the family $\{\rho_\ell\}_\ell$ (especially condition (3)). To do so, the Hodge conjecture will be used from now on.
\end{rmk}

Since $M$ is absolutely simple, $H_B(M_{| \C})$ is an irreducible Hodge structure. Fixing a polarization $\psi$ on $H_B(M_{| \C})$, the pair $(H_B(M_{| \C}), \psi)$ becomes $\Q$-PHS of K3 type. Moreover, since $\End (M)=\Q$, we have have that the endomorphism field of $(H_B(M_{| \C}), \psi)$ is $\Q$ as well (here again the fact that the representation is absolutely irreducible and the Hodge conjecture are fundamental).

Invoking the surjectivity of the period map, we want to produce a complex K3 surface from the rational polarised Hodge structure associated to $M_{|\C}$. We argue as follows.
\begin{prop}\label{lattices}
Let $(V,h,\psi)$ be a $\Q$-PHS of K3 type of dimension $22-\rho $. If $2\leq 22-\rho \leq 19$, then there exists a complex K3 surface $X$ with $T(X)_\Q$ isomorphic to $(V,h,\psi)$ as rational Hodge structures. 
\end{prop}
\begin{rmk}
Such proposition requires $\rho$ to be different from 1 and 2, where some restriction on the square class of the determinant of $(V,\psi)$ and its Hasse invariant appears (see \cite[Section IV]{omeara}). Even if the above proposition applies, the case $\rho = 20$ has to be excluded from the theorem, as remarked in section \ref{comments}. Indeed if $X$ has Picard rank 20 then the endomorphism field of $T(X)_\Q$ has to be larger than $\Q$ (see \cite[Remark 3.10 (page 54)]{bookk3} for an elementary proof of this fact). 
\end{rmk}
\begin{proof}[Proof of Proposition \ref{lattices}]
Notice that the quadratic form $(V,\psi)$ is rationally represented by the K3-lattice $(\Lambda_{\text{K3}},q)$. Indeed, as explained in \cite[Section IV]{omeara} (see also \cite[Theorems 17 and 31]{MR0037321}), this is true whenever
\begin{displaymath}
\text{defect}:=\dim \Lambda_{\text{K3}} - \dim V=\rho  \geq 3.
\end{displaymath}
We can therefore interpret $V$ as a subspace of $\Lambda_{\text{K3}}\otimes \Q$, and let $T$ be the intersection of $V$ with $\Lambda_{\text{K3}}$ (seen in $\Lambda_{\text{K3}}\otimes \Q$). By definition $T$ is a primitive sub-lattice of $\Lambda_{\text{K3}}$. Since the Hodge structure $h$ on $V$ is of K3 type, the quadratic form on $V$ (and thus on $T$) has signature $(2,19-\rho)$. Transporting the Hodge structure $h$ from $V$ to $T$ we obtain an irreducible integral Hodge structure with the right signature. Finally we can apply the surjectivity of the period map (see for example \cite[Theorem 6.3.1 and Remark 6.3.3 (page 114)]{bookk3}), to obtain a complex (algebraic) K3 surface $X$ such that $T(X)\cong T$.
\end{proof}

Let $X$ be the complex K3 surface obtained as in the proposition from $(H_B(M_{| \C}), \psi)$. Thanks to the Hodge conjecture we can lift the isomorphism of Hodge structures
\begin{displaymath}
T(X)_\Q \cong H_B(M_{|\C})
\end{displaymath}
to get an isomorphism at the level of motives. We indeed have
\begin{displaymath}
t^2(X)\cong M_{| \C} \in \mathcal{M}_\C,
\end{displaymath}
where $t^2(X)$ is the transcendental part of the motive of $X$, introduced in Section \ref{decmotivsurface}. 

To complete the proof we need a model $Y_L$ of $X$ defined over a finite extension $L$ of $K$, such that
\begin{displaymath}
t^2(Y_L)\cong M_{| L} \in \mathcal{M}_L.
\end{displaymath}
Since $M$ is defined over a number field, for all $\sigma \in \Aut(\C / \Qbar)$, we have the following chain of isomorphisms:
\begin{equation}\label{conj}
\leftidx{^\sigma}{t^2(X)}\cong \leftidx{^\sigma}{M_{| \C}}= M_{| \C} \cong t^2(X)  \in \mathcal{M}_\C.
\end{equation}
Notice that $\leftidx{^\sigma}{t^2(X)}=( \leftidx{^\sigma}{X}, \leftidx{^\sigma}{\pi}{^2_{alg}},0)= t^2(\leftidx{^\sigma}{X})$ from the uniqueness of the splitting $\pi^2= \pi^2_{alg}+\pi^2_{tr}$ in $X\times X$.

In particular we have
\begin{displaymath}
H_B(t^2(X))= T(X)_\Q  \text{   and   } H_B(\leftidx{^\sigma}{t^2(X)})= T(\leftidx{^\sigma}{X})_\Q.
\end{displaymath}

Taking the Betti realisation (with $\Q$-coefficients, as usual) of the equation \ref{conj}, we observe that $T(X)_\Q \cong T(\leftidx{^\sigma}{X})_\Q$ for all $\sigma\in  \Aut(\C / \Qbar)$. Applying Theorem \ref{main2}, that will be proved in the next section, this condition is enough to obtain a model $Y_L$ of $X/\C$ defined over some number field $L/K$ where $t^2(Y_L)\cong M_{| L}\in \mathcal{M}_L$.

Theorem \ref{mainthm} is finally proven: $Y_L$ is the K3 surface, defined over a finite extension $L$ of $K$, we were looking for. As remarked in the introduction we actually proved something more: there exists a simple motive $M$ defined over $K$ inducing the representations $\rho_\ell$, and a finite extension $L/K$ such that the base change of $M$ to $L$, denoted by $M_ L$, is isomorphic to the transcendental part of the motive of a K3 surface defined over $L$.

\section{Descent to a number field}\label{descentlemmas}
In this last section, we prove Theorem \ref{main2}. The result will follow from the combination of the following:
\begin{itemize}
\item The number of complex K3 surfaces, up to isomorphism, $Y$ such that $T(Y)_\Q$ is isomorphic to $T(X)_\Q$ is at most countable, cf. Lemma \ref{countable};
\item If all the conjugates of $X$, with respect to $\Aut(\mathbb{C} /  \Qbar)$, fall into countably many isomorphism classes, then $X$ descends to a number field, cf. Lemma \ref{descenttoqbar}.
\end{itemize}
The first point resembles the fact the the isogeny class of a given complex abelian variety consists of a countable set of (complex) abelian varieties (up to an isomorphism).

\begin{rmk} In the integral case we have the following. Let K3 be the full subcategory of complex varieties whose objects are K3 surfaces, and $X$ be a K3 surface. The set 
\begin{displaymath}
FM(X):=\{Y \in K3 \text{ such that there exists a Hodge isometry } T(Y) \cong T(X)\}
\end{displaymath}
contains only finitely many isomorphism classes. The proof of this result is due to Mukai, see \cite{mukai}, and builds on the derived Torelli theorem and the finiteness of the Fourier-Mukai partners. See also \cite[Proposition 16.3.10, Corollary 16.3.7 and Corollary 16.3.8]{bookk3} and \cite[Proposition 4.4]{martinskoro} for a direct argument which we emulate in the next lemma.
\end{rmk}

\begin{lemma}\label{countable}
Let $X/\C$ be a K3 surface. The set
\begin{displaymath}
S:=\{Y \in K3 \text{ such that } T(Y)_\Q\cong T(X)_\Q \text{ as }\Q-PHS\}/\text{isomorphism},
\end{displaymath}
is either finite or countable.
\end{lemma}
\begin{proof}
Let $Y\in S$, by reasoning as in \cite[Prop 4.4]{martinskoro}, it is enough to show that there are at most countably many choices for the rank and the discriminant of $T(Y)$. The rank is clearly fixed, so we have only to explain how the discriminant may vary. We notice that the discriminant of $T(Y)$ has to be equal to the discriminant of $T(X)$ modulo $(\Q^*)^2$, since the quadratic forms are non-degenerate, and so there are countably many choices.
\end{proof}

\begin{lemma}\label{descenttoqbar}
Let $X/\C$ be a K3 surface such that the set
\begin{displaymath}
\{\leftidx{^\sigma}{X}\}/\text{isomorphism}
\end{displaymath}
varying $\sigma \in \Aut(\mathbb{C} /  \Qbar)$ is at most countable. Then there exists a K3 surface $Y/\Qbar$ such that $Y \times_{\Qbar} \C$ is isomorphic to $X$.
\end{lemma}

\begin{proof}
This is by no means specific to K3 surfaces. Indeed it follows from a more general result due to Gonz\'{a}lez-Diez (and the fact that $\Qbar$ is countable). Let $X$ be an irreducible complex projective variety, in \cite[Criterion 1 (page 3)]{MR2253591} is proven that the following are equivalent:
\begin{itemize}
\item[a)] $X$ can be defined over $\Qbar$;
\item[b)] The set $\{\leftidx{^\sigma}{X}: \sigma \in \Aut(\mathbb{C} /  \Qbar)\}$ contains only finitely many isomorphism classes of complex projective varieties;
\item[c)] The set  $\{\leftidx{^\sigma}{X}: \sigma \in \Aut(\mathbb{C} /  \Qbar)\}$ contains only countably many isomorphism classes of complex projective varieties.
\end{itemize}
\end{proof}
\begin{rmk}\label{isogenyandabvar}
Using Gonz\'{a}lez-Diez's result as above, we can also offer another proof of \cite[Lemma 3.6]{MR3544295}, which, for example, does not  invoke the existence of the moduli space of abelian varieties (with some extra structure) established by Mumford in \cite[Part II, Section 6]{MR0204427}. Let $A/\C$ be an abelian variety such that all its $\Aut(\C / \Qbar)$-conjugates are isogenous to $A$. In particular the set $\{\leftidx{^\sigma}{A}: \sigma \in \Aut(\mathbb{C} /  \Qbar)\}$ is contained in the set of complex abelian varieties isogenous to $A$ which, up to isomorphism, is a countable set. As above, the implication c) $\Rightarrow$ a) shows that $A$ can be defined over $\Qbar$. 
\end{rmk}

\begin{proof}[Proof of Theorem \ref{main2}]
Thanks to Lemma \ref{countable}, we may apply Lemma \ref{descenttoqbar} which produces the desired model of $X$.
\end{proof}

\bibliographystyle{alpha}
\bibliography{biblio.bib}

\Addresses

\end{document}